\documentclass[12pt]{amsart}
\usepackage{amssymb,amsmath,amsthm,mathrsfs,multirow,xcolor,framed,url,tikz,xcolor,comment}
\usepackage[pdftex,
         pdfauthor={Dandan Chen and Jiahao Liu},
         pdftitle={Combinatorial Proofs of Overpartitions and Two-Color Partitions},
         pdfsubject={MATHEMATICS},
         pdfkeywords={Two-Color Partitions, Combination},
         pdfproducer={Latex with hyperref},
         pdfcreator={pdflatex}]{hyperref}
\definecolor{myyellow}{HTML}{EBE050} 
\definecolor{mypurple}{HTML}{9F86E7} 
\definecolor{myred}{HTML}{EB596E} 

\definecolor{mygreen}{HTML}{75F02B} 
\definecolor{myblue}{HTML}{0F88EB} 
\definecolor{mybrown}{HTML}{A52A2A} 
\definecolor{mygray}{HTML}{4b4453} 
\newtheorem{theorem}{Theorem}[section]
\newtheorem{lemma}[theorem]{Lemma}

\theoremstyle{definition}
\newtheorem{definition}[theorem]{Definition}

\theoremstyle{remark}

\numberwithin{equation}{section}

\newcommand\nutwid{\overset {\text{\lower 3pt\hbox{$\sim$}}}\nu}

\newcommand{\s}{\text{s}}

\allowdisplaybreaks

\newcommand\omycite[1]{}

\newcommand{\beqs}{\begin{equation*}}
\newcommand{\eeqs}{\end{equation*}}
\newcommand{\beq}{\begin{equation}}
\newcommand{\eeq}{\end{equation}}

\begin{document}
\title[Combinatorial proof ]{Combinatorial Proofs for Overpartitions and Two-Colored Partitions II}

\author{Dandan Chen}
\address{Department of Mathematics, Shanghai University, People's Republic of China}
\address{Newtouch Center for Mathematics of Shanghai University, Shanghai, People's Republic of China}
\email{mathcdd@shu.edu.cn}
\author{Jiahao Liu}
\address{Department of Mathematics, Shanghai University, People's Republic of China}
\email{2530476100@shu.edu.cn}


\date{}

	

\begin{abstract}
Andrews and El Bachraoui recently studied various two-colored integer partitions,
including those related to two-colored partitions into distinct parts with constraints and overpartitions.
Their work raised questions about the existence of combinatorial proofs for these results, which were partially addressed by the first author and Zou.
This paper  provides combinatorial proofs for the remaining results concerning two-colored partitions and overpartitions with constraints.

\end{abstract}

\maketitle

\section{Introduction}

In 2004, Corteel and Lovejoy \cite{CL04} introduced the concept of an overpartition, denoting by $\overline{p}(n)$  the number of overpartitions of $n$. An overpartition is defined as a partition of $n$ in which the first occurrence of each distinct part may optionally be overlined. From this, Hirschhorn and Sellers \cite{HS06} derived the generating function for $\overline{p_o}(n)$, the number of overpartitions of $n$ into odd parts, given by
\[
\sum_{n=0}^{\infty}\overline{p_o}(n)q^n=\frac{(-q;q^2)_{\infty}}{(q;q^2)_{\infty}},\quad \text{for $|q|<1$},
\]
where the $q$-shifted factorial \cite{A98} is defined as
\[
(a;q)_{\infty}=\prod_{j=0}^{\infty}(1-aq^j).
\]

\indent In recent research, Andrews and El Bachraoui \cite{AB24} studied two-colored partitions with specific constraints. They defined \(E(n)\) as the number of two-colored partitions of \(n\) where all parts are distinct, with the further requirement that even parts are confined to the blue color. They  further defined the following notations:

\begin{enumerate}
\item \(E_0(n)\) (resp. \(E_1(n)\)) indicates the number of such partitions of \(n\) where the count of even parts is even (resp. odd).
\item \(E_2(n)\) (resp. \(E_3(n)\)) indicates the number of such partitions of \(n\) where the total count of parts is even (resp. odd).
\end{enumerate}

In this paper, we  consider integer partitions with two colors, blue and green. A part \(\lambda_b\) (resp. \(\lambda_g\))  denotes a part \(\lambda\) occurring in blue (resp. green) color, following the order convention \(\lambda_b \geq \lambda_g\).\\
\indent For example, for $n=5$ we have $\overline{p}_o(5)=8$, counting the odd overpartitions:
\[\overline{5},\quad 5,\quad \overline{3}+\overline{1}+1, \quad \overline{3}+1+1,\quad 3+\overline{1}+1,\quad 3+1+1, \quad \overline{1}+1+1+1+1,\quad 1+1+1+1+1.\]
Moreover, we have $E(5) = 8$, which counts the two-colored partitions:
\[5_b, \quad 5_g,\quad 4_b+1_b, \quad4_b+1_g,\quad 3_b+2_b,\quad 3_g+2_b,\quad 3_b+1_b+1_g,\quad 3_g+1_b+1_g.\]
 Among these, $E_0(5)=E_1(5)=4$ and $E_2(5)=E_3(5)=4$.

\begin{theorem}\cite[Theorem 1]{AB24}\label{E}
For any nonnegative integer $n$, there holds
\begin{align*}
(a) & \quad E(n) = \overline{p_o}(n)\label{a}, \\
(b) & \quad E_0(n) = \begin{cases}
\frac{\overline{p_o}(n)}{2} + 1 & \text{if } n \text{ is a square}, \\
\frac{\overline{p_o}(n)}{2} & \text{otherwise},
\end{cases} \\
(c) & \quad E_1(n) = \begin{cases}
\frac{\overline{p_o}(n)}{2} - 1 & \text{if } n \text{ is a square}, \\
\frac{\overline{p_o}(n)}{2} & \text{otherwise},
\end{cases} \\
(d) & \quad E_2(n) = \begin{cases}
\frac{\overline{p_o}(n)}{2} + (-1)^n & \text{if } n \text{ is a square}, \\
\frac{\overline{p_o}(n)}{2} & \text{otherwise},
\end{cases} \\
(e) & \quad E_3(n) = \begin{cases}
\frac{\overline{p_o}(n)}{2} - (-1)^n & \text{if } n \text{ is a square}, \\
\frac{\overline{p_o}(n)}{2} & \text{otherwise}.
\end{cases}
\end{align*}
\end{theorem}

\indent In the following, we  provide combinatorial proofs for parts (b) through (e). Note that
\[
E(n) = \overline{p_o}(n),
\]
\[
E(n) = E_0(n) + E_1(n) = E_2(n) + E_3(n).
\]
Thus, it suffices to prove that
\begin{theorem}\label{Q}
For any nonnegative integer $n$, there holds
\begin{align*}
\text{(A)}\quad E_0(n)-E_1(n) &=
\begin{cases}
2 & \text{if } n \text{ is a square},\\
0 & \text{otherwise},
\end{cases} \\[1em]
\text{(B)}\quad E_2(n)-E_3(n) &=
\begin{cases}
2 \cdot(-1)^n & \text{if } n \text{ is a square},\\
0 & \text{otherwise}.
\end{cases}
\end{align*}
\end{theorem}

\textbf{Remark.} In  Theorem \ref{E},  a combinatorial proof of part (a) was previously established by Chen and Zou \cite{CZ25} through the construction of an explicit bijection. We  note that Bugleev \cite{Bugleev-arxiv} also combinatorially proved parts (b) through (e), but their proofs are different from ours.\\

\section{Combinatorial  proof of Theorem \ref{Q}}
As the proofs of Theorem \ref{Q} (A) and (B) are analogous, we first prove (A) and then (B).
From the definition of $E(n)$, any partition counted by $E(n)$ can be uniquely decomposed into three types of parts:
\begin{itemize}
\item Blue even parts (denoted by the set $\lambda_{\text{even}}$),
\item Green odd parts (denoted by the set $\alpha_{\text{odd}}$),
\item Blue odd parts (denoted by the set $\beta_{\text{odd}}$),
\end{itemize}
where all parts across these sets are distinct.

\indent Let $i,j,k \in \mathbb{N}$ denote the number of parts in $\lambda_{\text{even}}$, $\alpha_{\text{odd}}$, $\beta_{\text{odd}}$, respectively.
Now, suppose  a two-colored partition counted by $E(n)$ is given by the triple $(\lambda_{\text{even}}, \alpha_{\text{odd}}, \beta_{\text{odd}})$, where:
\begin{align*}
\lambda_{\text{even}}& = (\lambda_{1}, \lambda_{2}, \ldots, \lambda_{i}),\\
\alpha_{\text{odd}}& = (2\alpha_{1}+1, 2\alpha_{2}+1, \ldots, 2\alpha_{j}+1),\\
\beta_{\text{odd}}& = (2\beta_{1}+1, 2\beta_{2}+1, \ldots, 2\beta_{k}+1)
\end{align*}
represent the sequences of blue even parts,  green odd parts, and  blue odd parts, respectively, each  arranged in decreasing order.\\
\indent To proceed, it is expedient to introduce the notions of a \emph{bi-partition} and a \emph{system of parallel bi-partitions}.
\begin{definition}\cite[p.284]{JS12}
A \emph{bi-partition} of nonnegative integer $n$ is defined as a division of $n$ into two subsets of odd integers, denoted by $L$ and $R$.  In other words, a pair $(L,R)$ is a \emph{bi-partition} of $n$ if
\[
    \sum_{l \in L} l \;+\; \sum_{r \in R} r \;=\; n.
\]
\indent A \emph{parallel bi-partition system} of  $n$ is a  \emph{bi-partition} $(L,R)$ with the additional constraint that the cardinalities of  $L$ and $R$ differ by a fixed constant $c \geq 0$, i.e.,
\[
   \bigl|\,|L| - |R|\,\bigr| \;=\; c.
\]
\end{definition}

\indent We now associate the sets of odd parts with a \emph{bi-partition}. Specifically, we let the blue odd parts $\beta_{\text{odd}}$ correspond to the left part $L$, and the green odd parts $\alpha_{\text{odd}}$ correspond to the right part $R$. The constant $c$ for the resulting \emph{parallel bi-partition system} is then the absolute difference in the number of parts between these two sets, i.e. $c=\big| |\beta_{\text{odd}}|-|\alpha_{\text{odd}}| \big|$.

As an illustration, consider the partition $2_b + 5_b + 3_b + 1_b + 3_g$ counted by $E(14)$. The blue even part ($2_b$) is handled separately. The remaining odd parts form a \emph{bi-partition}. Graphically, each square represents a unit of $1$. The corresponding \emph{parallel bi-partition system} with cardinality difference $c=|3-1|=2$ is given by:

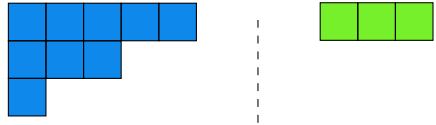
\[
\begin{tikzpicture}[scale=0.5, baseline=(current bounding box.center)]
    \fill[myblue] (1,0) rectangle ++(5,-1);
    \fill[myblue] (1,-1) rectangle ++(3,-1);
    \fill[myblue] (1,-2) rectangle ++(1,-1);

    \foreach \y/\n in {0/5,1/3,2/1} {
        \foreach \x in {1,...,\n} {
            \draw (\x,-\y) rectangle ++(1,-1);
        }
    }
    \node[below] at (3,-3.5) {$L=\beta_{\text{odd}}=(5_b,3_b,1_b)$};
\end{tikzpicture}
\;\;
\tikz[baseline]{\draw[dashed] (0,-1)--(0,1);}
\;\;
\begin{tikzpicture}[scale=0.5, baseline=(current bounding box.center)]
    \fill[mygreen] (1,0) rectangle ++(3,-1);
    \foreach \y/\n in {0/3} {
        \foreach \x in {1,...,\n} {
            \draw (\x,-\y) rectangle ++(1,-1);
        }
    }
    \node[below] at (3,-3.5) {$R=\alpha_{\text{odd}}=(3_g)$};
\end{tikzpicture}
\]

By symmetry, we may assume without loss of generality that $|L| - |R| = c \geq 0$.
\begin{lemma} \label{lem1}
Let $n \in \mathbb{N^+}$. Consider a parallel bi-partition system of $n$ composed of:
\[
R = \alpha_{\text{odd}} = (2\alpha_{1}+1, \dots, 2\alpha_{j}+1), \quad
L = \beta_{\text{odd}} = (2\beta_{1}+1, \dots, 2\beta_{k}+1)
\]
with $k-j=c\geq0$ and all parts distinct and postive. Define the associated integers:
\begin{align*}
d &= \sum_{i=1}^{j} \alpha_{i} + \sum_{i=1}^{k} (\beta_{i}+1),\\
t &= \sum_{i=1}^{j} (\alpha_{i}+1) + \sum_{i=1}^{k} \beta_{i}.
\end{align*}
Then, $n = d+t$ holds.
Moreove, for a fixed difference $c$ and a fixed value of $d$, the number of such parallel bi-partition systems corresponding to $(\alpha_{\text{odd}}, \beta_{\text{odd}})$ is given by
\[
p\!\left(d - \tfrac{1}{2}c(c+1)\right),
\]
where $p(m)$ denotes the number of integer partitions of $m$.
\end{lemma}

\begin{proof}
 We construct the concatenation diagram according to the method in \cite[p.56]{EW65} as follows:  (1) The major half of $\beta_{\text{odd}}$, defined as the sequence $\beta'_{\text{odd}}=(\beta_1+1, \beta_2+1, \ldots, \beta_k+1)$. The minor half of $\alpha_{\text{odd}}$, defined as the sequence $\alpha'_{\text{odd}}=(\alpha_1, \alpha_2, \ldots, \alpha_j)$. (2) Represent the sequence   $(\beta_{1}+1,\ldots,\beta_{k}+1)$  vertically. Each entry shifted one unit downward.
(3) Represent the sequence $(\alpha_{1},\ldots,\alpha_{j})$ horizontally. Each entry shortened by one unit.
(4) These two parts are joined after inserting exactly \(c=k-j\) empty rows.

The following examples illustrate this construction.

\textbf{Example 1.} Let
$\beta_{\text{odd}} = (9_b, 5_b, 3_b, 1_b)$ ($k=4$) and $\alpha_{\text{odd}} = (7_g, 1_g)$ ($j=2$). So $c=2$.
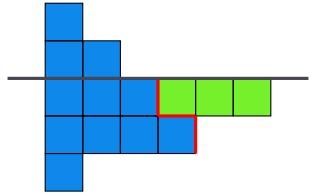
\[
\begin{tikzpicture}[scale=0.5, baseline=(current bounding box.center)]
    \fill[myblue] (1,0) rectangle ++(1,-1);
    \fill[myblue] (1,-1) rectangle ++(2,-1);
    \fill[myblue] (1,-2) rectangle ++(3,-1);
    \fill[myblue] (1,-3) rectangle ++(4,-1);
    \fill[myblue] (1,-4) rectangle ++(1,-1);

    \fill[mygreen] (4,-2) rectangle ++(3,-1);

    \foreach \y/\n in {0/1,1/2,2/6,3/4,4/1}
    {
        \foreach \x in {1,...,\n} {
            \draw (\x,-\y) rectangle ++(1,-1);
        }
    }

    \draw[red, very thick] (4,-2) -- (4,-3); 
    \draw[red, very thick] (4,-3) -- (5,-3);
    \draw[red, very thick] (5,-3) -- (5,-4);

    \draw[mygray, very thick] (0,-2) -- (8,-2); 
\end{tikzpicture}
\]

\[
\beta'_{\text{odd}}=(5_b, 3_b, 2_b, 1_b),\quad \alpha'_{\text{odd}} = (3_g)
\]

\textbf{Note:} In the figure, when the minor half of $1$ leaves a gap at the red joining line, it is denoted as $1_{g}$.

\indent \textbf{Example 2.} Let $\beta_{\text{odd}} = (13_b, 9_b, 5_b, 3_b, 1_b)$ ($k=5$) and $\alpha_{\text{odd}} = (9_g, 7_g, 3_g)$ ($j=3$). So $c=2$.
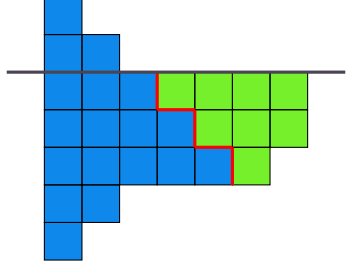
\[
\begin{tikzpicture}[scale=0.5, baseline=(current bounding box.center)]
    \fill[myblue] (1,0) rectangle ++(1,-1);
    \fill[myblue] (1,-1) rectangle ++(2,-1);
    \fill[myblue] (1,-2) rectangle ++(3,-1);
    \fill[myblue] (1,-3) rectangle ++(4,-1);
    \fill[myblue] (1,-4) rectangle ++(5,-1);
    \fill[myblue] (1,-5) rectangle ++(2,-1);
    \fill[myblue] (1,-6) rectangle ++(1,-1);

    \fill[mygreen] (4,-2) rectangle ++(4,-1);
    \fill[mygreen] (5,-3) rectangle ++(3,-1);
    \fill[mygreen] (6,-4) rectangle ++(1,-1);

    \foreach \y/\n in {0/1,1/2,2/7,3/7,4/6,5/2,6/1}
    {
        \foreach \x in {1,...,\n} {
            \draw (\x,-\y) rectangle ++(1,-1);
        }
    }

    \draw[red, very thick] (4,-2) -- (4,-3);
    \draw[red, very thick] (4,-3) -- (5,-3);
    \draw[red, very thick] (5,-3) -- (5,-4);
    \draw[red, very thick] (5,-4) -- (6,-4);
    \draw[red, very thick] (6,-4) -- (6,-5);

    \draw[mygray, very thick] (0,-2) -- (9,-2); 
\end{tikzpicture}
\]
\[
\beta'_{\text{odd}}=(7_b, 5_b, 3_b, 2_b, 1_b),\quad \alpha'_{\text{odd}} = (4_g, 3_g, 1_g)
\]

\indent Analyzing these figures, we can divide the concatenated diagram into two region: the upper region, which is a fixed  triangular number $T(c)=c(c+1)/2$ (determined by the shift and the $c$ empty rows),  and the lower region, which represents an unrestricted partition of the remaining squares.

The  process described above is reversible. Starting with a constant $c$, we determine the corresponding triangular number, then give the unrestricted partition, restoring the concatenated diagram to a \emph{parallel bi-partition system}, leading to the two-colored partition. Therefore, this construction establishes  a bijection.\\
\indent Recall that the total number $n$ is given by the sum of all parts:
\[
n = \sum_{i=1}^{j} (2\alpha_{i}+1) + \sum_{i=1}^{k} (2\beta_{i}+1).
\]
The quantities $d$ and $t$ are defined as:
\[
d = \sum_{i=1}^{j} \alpha_{i} + \sum_{i=1}^{k} (\beta_{i}+1),
\qquad
t = \sum_{i=1}^{j} (\alpha_{i}+1) + \sum_{i=1}^{k} \beta_{i}.
\]
It is straightforward to verify  that $n = d + t$. In our geometric interpretation,  $d$ represents the total number of squares in the concatenated diagram.

Given $c$, an unrestricted partition of $d - \frac{1}{2}c(c+1)$ uniquely determines the concatenated diagram and the corresponding two-colored partition.
Hence, for fixed parameters $c$ and $d$, the number of \emph{parallel bi-partition systems} with a constant difference $c$ for $n$ is exactly
\[
p\!\left(d - \tfrac{1}{2}c(c+1)\right).
\]
\end{proof}

The following lemma, a classical result due to Euler \cite[Corollary 1.8]{A98}, provides a recurrence relation for the partition function $p(n)$ in terms of generalized pentagonal numbers.
It will play a crucial role in our combinatorial analysis of the contributions  from  even and odd parts in two-colored partitions.
\begin{lemma}\cite[Corollary 1.8(Euler)]{A98} \label{lem2}
Let $n\in\mathbb{N^+}$. Then
\begin{align*}
0 =\;& p(n) - p(n-1) - p(n-2) + p(n-5) + p(n-7) - \cdots \\[4pt]
&\; + (-1)^m\, p\!\left(n - \tfrac{1}{2}m(3m-1)\right) + (-1)^m\, p\!\left(n - \tfrac{1}{2}m(3m+1)\right) + \cdots
\end{align*}
where we recall that $p(0)=1$ and $p(M)=0$ for all negative integers $M$.
\end{lemma}

\begin{proof}[The proof of Theorem 1.2]
First, we construct the Franklin involution \cite[Theorem 1.6]{A98} specifically on the set of  blue even parts, denoted $\lambda_{\text{even}}$, within a two-colored partition counted by $E(n)$. This involution will be used to cancel out  certain partitions in a reversing manner.\\
\indent For a partition $\lambda_{\text{even}}=(\lambda_1,\lambda_2,\cdots,\lambda_i)$ (in decreasing order), let $s(\lambda) = \lambda_{i}$ denote its smallest part.
Furthermore,  the largest part $\lambda_{1}$ begins a sequence of consecutive even integers.
 We denote the length of  this consecutive sequence by $\sigma(\lambda)$, defined as  the largest integer $j$ such that $\lambda_{j} = \lambda_{1} - j + 1$ for $1\leq j\leq i$.\\
\indent Graphically, we perform a modulo-2 partition on each blue positive even part, where each circle represents a $2$. The parameters $s(\lambda)$ and $\sigma(\lambda)$ can then be illustrated as follows.
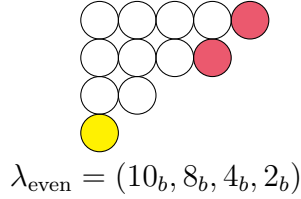
\[
\begin{tikzpicture}[scale=0.5, baseline=(current bounding box.center)]
    \fill[yellow] (1.5,-3.5) circle (0.5);
    \fill[myred] (4.5,-1.5) circle (0.5);
    \fill[myred] (5.5,-0.5) circle (0.5);

    \foreach \y/\n in {0/5,1/4,2/2,3/1} {
        \foreach \x in {1,...,\n} {
            \draw (\x+0.5,-\y-0.5) circle (0.5);
        }
    }
    \node[below] at (3,-4) {$\lambda_{\text{even}} = (10_b, 8_b, 4_b, 2_b)$};
\end{tikzpicture}
\]

\indent In the diagram, the yellow circle represents $s(\lambda)$ and the red circle represents $\sigma(\lambda)$.\\
\indent The transformation (Franklin involutions) is then applied according to the following rule:
\begin{itemize}
\item[Case 1:] If \(\s(\lambda) \leq \sigma(\lambda)\), adding one to each of the \(\s(\lambda)\) largest parts of \(\lambda\) and deleting the smallest part.

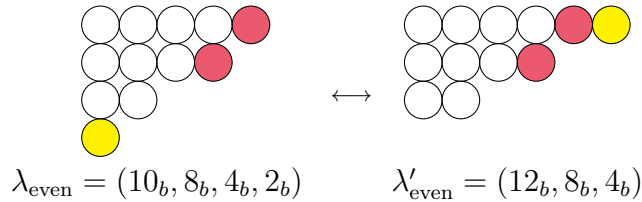
\[
\begin{tikzpicture}[scale=0.5, baseline=(current bounding box.center)]
    \fill[yellow] (1.5,-3.5) circle (0.5);
    \fill[myred] (4.5,-1.5) circle (0.5);
    \fill[myred] (5.5,-0.5) circle (0.5);

    \foreach \y/\n in {0/5,1/4,2/2,3/1}
    {
        \foreach \x in {1,...,\n} {
            \draw (\x+0.5,-\y-0.5) circle (0.5);
        }
    }

    \node[below] at (3,-4) {$\lambda_{\text{even}} = (10_b, 8_b, 4_b, 2_b)$};
\end{tikzpicture}
\;\;\overset{\longleftrightarrow}\;\;
\begin{tikzpicture}[scale=0.5, baseline=(current bounding box.center)]
    \fill[yellow] (6.5,-0.5) circle (0.5);
    \fill[myred] (4.5,-1.5) circle (0.5);
    \fill[myred] (5.5,-0.5) circle (0.5);

    \foreach \y/\n in {0/6,1/4,2/2}
    {
        \foreach \x in {1,...,\n} {
            \draw (\x+0.5,-\y-0.5) circle (0.5);
        }
    }

    \node[below] at (4,-4) {$\lambda'_{\text{even}} = (12_b, 8_b, 4_b)$};
\end{tikzpicture}
\]

\item[Case 2:] If \(\s(\lambda) > \sigma(\lambda)\), subtracting one from each of the \(\sigma(\lambda)\) largest parts of \(\lambda\) and inserting a new smallest part of size \(\sigma(\lambda)\).

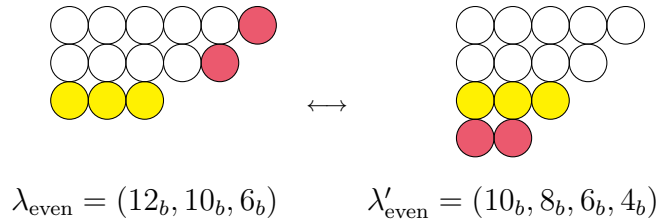
\[
\begin{tikzpicture}[scale=0.5, baseline=(current bounding box.center)]
    \fill[yellow] (1.5,-2.5) circle (0.5);
    \fill[yellow] (2.5,-2.5) circle (0.5);
    \fill[yellow] (3.5,-2.5) circle (0.5);
    \fill[myred] (5.5,-1.5) circle (0.5);
    \fill[myred] (6.5, -0.5) circle (0.5);

    \foreach \y/\n in {0/6,1/5,2/3}
    {
        \foreach \x in {1,...,\n} {
            \draw (\x+0.5,-\y-0.5) circle (0.5);
        }
    }

    \node[below] at (3.5,-4.5) {$\lambda_{\text{even}} = (12_b, 10_b, 6_b)$};
\end{tikzpicture}
\;\;\overset{\longleftrightarrow}\;\;
\begin{tikzpicture}[scale=0.5, baseline=(current bounding box.center)]
    \fill[yellow] (1.5,-2.5) circle (0.5);
    \fill[yellow] (2.5,-2.5) circle (0.5);
    \fill[yellow] (3.5,-2.5) circle (0.5);
    \fill[myred] (1.5,-3.5) circle (0.5);
    \fill[myred] (2.5,-3.5) circle (0.5);

    \foreach \y/\n in {0/5,1/4,2/3,3/2}
    {
        \foreach \x in {1,...,\n} {
            \draw (\x+0.5,-\y-0.5) circle (0.5);
        }
    }

    \node[below] at (2.5,-4.5) {$\lambda'_{\text{even}} = (10_b, 8_b, 6_b, 4_b)$};
\end{tikzpicture}
\]

\end{itemize}

\indent A difficulty arises when the yellow and red segments, denoted by $s(\lambda)$ and $\sigma(\lambda)$, overlap.
In fact, the procedure remains valid in all cases except when $s(\lambda) = \sigma(\lambda)$ or $s(\lambda) = \sigma(\lambda)+1$,
which correspond to the values  $\frac{1}{2}m(3m-1)$ and $\frac{1}{2}m(3m+1)$, respectively.
We  therefore refer to such instances as the ``pentagonal case''.
For example, this occurs when $s(\lambda) = \sigma(\lambda) = 3$ or $s(\lambda) = 3, \, \sigma(\lambda) = 2$.

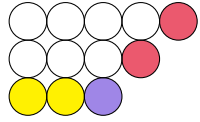
\[
\begin{tikzpicture}[scale=0.5, baseline=(current bounding box.center)]
    \fill[yellow] (1.5,-2.5) circle (0.5);
    \fill[yellow] (2.5,-2.5) circle (0.5);
    \fill[mypurple] (3.5,-2.5) circle (0.5);
    \fill[myred] (4.5,-1.5) circle (0.5);
    \fill[myred] (5.5,-0.5) circle (0.5);

    \foreach \y/\n in {0/5,1/4,2/3}
    {
        \foreach \x in {1,...,\n} {
            \draw (\x+0.5,-\y-0.5) circle (0.5);
        }
    }

    \node[below] at (3,-3.5) {$s(\lambda)=\sigma(\lambda)=3$};
\end{tikzpicture}
\;\;
\hspace{2cm} 
\;\;
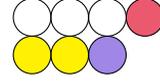
\begin{tikzpicture}[scale=0.5, baseline=(current bounding box.center)]
    \fill[yellow] (1.5,-1.5) circle (0.5);
    \fill[yellow] (2.5,-1.5) circle (0.5);
    \fill[mypurple] (3.5,-1.5) circle (0.5);
    \fill[myred] (4.5,-0.5) circle (0.5);

    \foreach \y/\n in {0/4,1/3}
    {
        \foreach \x in {1,...,\n} {
            \draw (\x+0.5,-\y-0.5) circle (0.5);
        }
    }

    \node[below] at (2.5,-3.5) {$s(\lambda)=3,\ \sigma(\lambda)=2$};
\end{tikzpicture}
\]

\indent Up to this point,  we have established a one-to-one correspondence-based on their even parts-between partitions in $E_0(n)$ and $E_1(n)$ that do not fall into the pentagonal case. Therefore, the subsequent analysis need only focus on partitions whose even parts fall into the pentagonal cases.\\
\indent Applying Lemma \ref{lem1}, we let $c$ and $d$ be given. Set $n=d+t$, where $t$ is chosen such that $d-t=c$. For the even parts in the pentagonal number case, the total sum of the even parts is $m(3m \pm 1)$, where $m$ denotes the number of even parts. Consequently, the sum of the elements in the parallel bi-partitions corresponding to the two-colored partitions of the odd parts is
\[
n - m(3m \pm 1)
= d'+t',
\]
where
\[
d' = d - \tfrac{1}{2}m(3m \pm 1),
\quad
t' = t - \tfrac{1}{2}m(3m \pm 1)
\]
satisfying $d'-t'=c$.
Then the number of \emph{parallel bi-partition systems} corresponding to $(\alpha_{\text{odd}}, \beta_{\text{odd}})$ is
\[
p\!\left(d' - \tfrac{1}{2}c(c+1)\right)
= p\!\left(d - \tfrac{1}{2}c(c+1) - \tfrac{1}{2}m(3m \pm 1)\right).
\]
Subtracting the odd-part case from the even-part case yields the following expression:
\begin{align*}
& p(d - \tfrac{1}{2}c(c+1)) - p(d - \tfrac{1}{2}c(c+1)-1) - p(d - \tfrac{1}{2}c(c+1)-2) + \cdots \\[4pt]
&\; + (-1)^m\, p\!\left(d - \tfrac{1}{2}c(c+1) - \tfrac{1}{2}m(3m-1)\right) + (-1)^m\, p\!\left(d - \tfrac{1}{2}c(c+1) - \tfrac{1}{2}m(3m+1)\right) + \cdots.
\end{align*}
Using Lemma \ref{lem2}, this expression equals zero unless $d - \tfrac{1}{2}c(c+1) = 0$. In that case,  $n=d+t=c^2$ and $p(0)=1$.
By symmetry, if $d = \tfrac{1}{2}c(c-1)$, then $t = \tfrac{1}{2}c(c+1)$, so again $n=d+t=c^2$ and $p(0)=1$.    This completes the proof of  Theorem \ref{Q} (A).\\
\indent For the proof of Theorem \ref{Q} (B), we  proceed by applying the same method and conclusions as above.
First,  using Franklin involution \cite[Theorem 1.6]{A98}, we eliminate the general two-colored partitions, leaving only those partitions whose even parts fall into the pentagonal cases.
Then, by Lemma \ref{lem1}, subtracting the odd-part case from the even-part case yields the following expression:
 \begin{align*}
& (-1)^n \Bigg[ \, p\!\left(d - \tfrac{1}{2}c(c+1)\right) - p\!\left(d - \tfrac{1}{2}c(c+1)-1\right) - p\!\left(d - \tfrac{1}{2}c(c+1)-2\right) + \cdots \\[4pt]
&\quad + (-1)^m\, p\!\left(d - \tfrac{1}{2}c(c+1) - \tfrac{1}{2}m(3m-1)\right) + (-1)^m\, p\!\left(d - \tfrac{1}{2}c(c+1) - \tfrac{1}{2}m(3m+1)\right) + \cdots \, \Bigg].
\end{align*}
In this expression, the factor $(-1)^{m}$ accounts for the parity of the number of even parts, while the factor $(-1)^{n}$ (under the condition $n\equiv c\pmod2$) accounts for the parity of  the overall number of odd parts.
Finally,  following the same line of argument as above, we obtain the proof of Theorem \ref{Q} (B).

\end{proof}

\subsection*{Acknowledgements}
We appreciate Professor John Loxton for bringing Bugleev's paper to our attention. The first author was  supported by the National Key R\&D Program of China (Grant No. 2024YFA1014500) and the National Natural Science Foundation of China (Grant No. 12201387).


\begin{thebibliography}{10}
\bibitem{A98} G. E. Andrews, \emph{The Theory of Partitions}, Cambridge University Press, Cambridge, 1998.

\bibitem{AB24} G. E. Andrews and M. El Bachraoui, On two-color partitions with odd smallest part, arXiv:2410.14190.

\bibitem{Bugleev-arxiv}A. Bugleev, TWO-COLOR PARTITIONS AND OVERPARTITIONS: A COMBINATORIAL PROOF, arXiv:2508.19741v1


\bibitem{CZ25} D. Chen and Z. Zou, COMBINATORIAL PROOFS FOR TWO-COLOR PARTITIONS. Bull. Aust. Math. Soc., doi:10.1017/S0004972725000206.

\bibitem{CL04} S. Corteel and J. Lovejoy, Overpartitions, Trans. Amer. Math. Soc. 356 (2004), 1623--1635.

\bibitem{HS06} M. D. Hirschhorn and J. A. Sellers, Arithmetic properties of overpartitions into odd parts, Ann. Comb. 10 (2006), 353--367.

\bibitem{JS12} Sylvester, J. J.; Franklin, F. \newblock ``A Constructive Theory of Partitions Arranged in Three Acts, an Interact, and an Exodion.''\newblock {\em Amer. J. Math.} {\bf 5}(1882), 251--330.

\bibitem{EW65} E. M. Wright, An enumerative proof of an identity of Jacobi, J. London Math. Soc. 40 (1965), 55-57.
\end{thebibliography}
\end{document}